\documentclass[twoside, 12pt]{article}
\usepackage{mathrsfs}
\usepackage{amssymb, amsmath, mathrsfs, amsthm}
\usepackage{graphicx}
\usepackage{color}
\usepackage[top=2.5cm, bottom=2.5cm, left=2.5cm, right=2.5cm]{geometry}
\usepackage{float, caption, subcaption}

\input{epsf}

\newcommand{\bd}{\begin{description}}
\newcommand{\ed}{\end{description}}
\newcommand{\bi}{\begin{itemize}}
\newcommand{\ei}{\end{itemize}}
\newcommand{\be}{\begin{enumerate}}
\newcommand{\ee}{\end{enumerate}}
\newcommand{\beq}{\begin{equation}}
\newcommand{\eeq}{\end{equation}}
\newcommand{\beqs}{\begin{eqnarray*}}
\newcommand{\eeqs}{\end{eqnarray*}}

\newcommand{\Rmnum}[1]{\expandafter\@slowromancap\romannumeral #1@}

\definecolor{DarkGreen}{rgb}{0.2, 0.6, 0.3}

\newtheorem{theorem}{Theorem}[section]

\newtheorem{lemma}{Lemma}[section]
\theoremstyle{definition}

\newtheorem{claim}{Claim}
\newtheorem{remark}{Remark}[section]

\newtheorem{proposition}{Proposition}[section]

\begin{document}
\title{Maximum packings in graphs forbidding given rainbow cycles}
\author{\small Ping Li, Yang Yang\\
{\small School of Mathematics and Statistics, Shaanxi Normal University, Xi'an, Shaanxi 710062}\\
{\small  Emails: lp-math@snnu.edu.cn, yyang$\_\_$@snnu.edu.cn}\\
}
\date{}
\maketitle

\begin{abstract}
For graphs $F$ and $G$, $F$-multicolor Tur\'{a}n number of $G$, denoted by $\mathrm{ex}_F(n,G)$, is the maximum number of edge-disjoint copies of $F$ in an $n$-vertex graph such that there is no copy of $G$ whose edges come from distinct copies of $F$. We study this parameter mainly for cycle pairs and determine, up to asymptotic order, when $\mathrm{ex}_{C_k}(n,C_\ell)$ attains the three natural thresholds: the upper bound, the lower bound, and the $n^{2-o(1)}$ regime. In particular, for every odd $k\ge 5$ and every $t\ge 1$, where $C_k(t)$ denotes the $t$-blow-up of $C_k$, we prove $\mathrm{ex}_{C_k(t)}(n,C_{k-2})=n^2/(kt)^2+o(n^2),$
and establish a corresponding stability theorem. We further show that if $F$ and $G$ have the same odd girth $k$ and there exist  homomorphisms from both $F$ and $G$ to $C_k$, then $\mathrm{ex}_F(n,G)=n^{2-o(1)}$; in particular, $\mathrm{ex}_{C_k}(n,C_k)=n^{2-o(1)}$ for odd $k$. In addition, we prove $\mathrm{ex}_{C_{2k+1}}(n,C_{2\ell+1})=O\!\left(n^{1+1/\lceil \ell/k\rceil}\right)$ for $\ell>k$ and  $\mathrm{ex}_F(n,G)=O(\mathrm{ex}(n,G))$ for bipartite $G$.
We particularly establish $\mathrm{ex}_{C_4}(n,C_4)=\frac{\sqrt{2}}{8}n^{3/2}+O(n)$, and give a sufficient condition under which the lower bound cannot be attained.\\[2mm]
{\bf Keywords:} $F$-multicolor Tur\'{a}n number, cycles, the lower bound, the $(6,3)$-type bound, stability  \\[2mm]
{\bf AMS subject classification 2020:} 05C35.
\end{abstract}

\section{Introduction}

Tur\'{a}n-type problems are central objects in
extremal graph theory: one seeks the
maximum number of edges in an $n$-vertex
graph avoiding a prescribed subgraph. Various extensions and variants of the problem have significantly enriched the field. In this paper, we study a new variant of Tur\'{a}n-type problems whose definition is closely related to the well-known Ruzsa-Szemer\'{e}di $(6,3)$-problem \cite{6-3}.

The Ruzsa-Szemer\'{e}di $(6,3)$-problem (or $(6,3)$-problem for short) can be equivalently stated as follows: determine the maximum number of edge-disjoint triangles in an $n$-vertex graph such that no triangle is formed by taking one edge from each of three distinct triangles. Imolay, Karl, Nagy and V\'{a}li \cite{Imolay} generalized this problem by considering the following: given two fixed graphs $F$ and $G$, determine the maximum number of edge-disjoint copies of $F$ in an $n$-vertex graph such that there is no copy of $G$ whose edges come from distinct copies of $F$. This maximum, denoted by $ex_F(n,G)$, is called the \emph{$F$-multicolor Tur\'{a}n number} of $G$ (or {\em multicolor Tur\'{a}n number} of $G$ for short). When $F=K_2$, this notion reduces to the classical Tur\'{a}n problem for the forbidden subgraph $G$.

Determining $ex_F(n,G)$ is difficult even for small cycle pairs, as illustrated by the
$(6,3)$-problem \cite{6-3} and the still-nontrivial case $ex_{C_5}(n,C_3)$ \cite{Balogh}.

For convenience, we may regard a collection of edge-disjoint copies of $F$ as a family of monochromatic copies of $F$, each assigned a distinct color. Any such edge-disjoint copy of $F$ is called a {\em monochromatic $F$-copy}, while any other (not necessarily monochromatic) copy of $F$ is simply called a {\em copy of $F$}.
Then $ex_F(n,G)$ is the maximum number of monochromatic $F$-copies in an $n$-vertex graph that contains no rainbow copy of $G$. 

The $F$-multicolor Tur\'{a}n number is a very interesting concept, as it is closely related to three well-studied problems: the linear Tur\'{a}n number of Berge graphs, the packing number, and the generalized Tur\'{a}n number.

\subsection{Two related problems}

For two $r$-uniform hypergraphs $G$ and $F$, the \emph{$G$-packing number} of a graph $H$, denoted $\nu_H(G)$, is the maximum number of edge-disjoint copies of $G$ in $H$. An interesting problem is to determine the packing number in graphs that forbid some given graph. We use $\nu_n(F,G)$ to denote the maximum $F$-packing number among all $n$-vertex $G$-free graphs. It is worth noting that $ex_F(n,G)$ and $\nu_n(F,G)$ are deeply related.

\begin{proposition}\label{prop-1}
For any two $r$-uniform hypergraphs $G$ and $F$, $\nu_n(F,G)=ex_F(n,G)+o(n^r)$.
\end{proposition}

The second related problem is the generalized Tur\'{a}n number $ex(n,F,G)$, which is the maximum number of copies of $F$ in an $n$-vertex $G$-free graph. The case $F=C_5$ and $G=C_3$ is the well-known pentagonal conjecture of Erd\H{o}s \cite{Erdos-53}, which was solved independently by Grzesik \cite{Grzesik-1} and Hatami, Hladk\'{y}, Kr\'{a}l, Norine, and Razborov \cite{Hatami} using flag algebras. Grzesik and Kielak \cite{Grzesik} generalized the result by giving an asymptotic value of $ex(n,C_{2r+1},C_{2r-1})$ for each $r\geq 2$. Beke and Janzer \cite{Beke} further characterized the extremal constructions. Bollob\'{a}s and Gy\H{o}ri \cite{Bollobas} considered the opposite problem $ex(n,C_3,C_5)$, and Gishboliner and Shapira \cite{Gishboliner} studied further cases of ``cycles versus cycles''. The systematic study of generalized Tur\'{a}n
numbers was initiated by Alon and Shikhelman \cite{Alon-S}. For a comprehensive survey on generalized Tur\'{a}n problems, we refer the reader to \cite{Gebener}.

\subsection{A generalization of the linear Tur\'{a}n number of Berge graphs}

Apart from the packing number and the generalized Tur\'an number, the third related problem is the linear Tur\'an number of Berge graphs. Our first main result extends some linear Tur\'an numbers of Berge cycles in $3$-uniform graphs.

For a graph $G$, the \emph{linear Tur\'an number} of Berge-$G$, denoted $ex_r^{lin}(n,G)$, is the maximum number of edges in an $n$-vertex linear $r$-uniform hypergraph containing no  Berge-$G$ (i.e., no copy of $G$ whose edges come from distinct hyperedges). If $F=K_r$ is a complete graph, then any two monochromatic $F$-copies intersect in at most one vertex. Regarding each such monochromatic $F$-copy as an $r$-uniform edge, we obtain $ex_F(n,G)=ex_r^{\operatorname{lin}}(n,G)$.

For linear Tur\'an numbers of Berge cycles hosted in $3$-uniform graphs, 
Ergemlidze, Gy\H{o}ri and Methuku \cite{Ergemlidze} proved that 
$ex_{C_3}(n,C_4)= n^{3/2}/6+O(n)$ and $ex_{C_3}(n,C_5)= n^{3/2}/(3\sqrt{3})+O(n)$;
F\"{u}redi and \"{O}zkahya \cite{Furedi} established the general bound 
$ex_{C_3}(n,C_{2k+1})\le 2kn^{1+1/k}+9kn$;
Collier-Cartaino, Graber, and Jiang \cite{Collier-Cartaino} extended the result to 
$ex_{C_3}(n,C_k)\le O(n^{1+1/\lfloor k/2\rfloor})$.
In the following we extend the above results. In particular, the second result generalizes those of \cite{Collier-Cartaino} and \cite{Furedi} with respect to the exponents.

\begin{theorem}\label{thm-4}
    Let $k,\ell\geq 1$be two integers.
    \begin{itemize}
        \item [(1).] If $F$ is a nonempty graph and $G$ is a bipartite graph, then $ex_F(n,G)=O(ex(n,G))$. Specifically, $ex_{F}(n,C_{2\ell})\leq O(n^{1+1/\ell})$.
        \item [(2).] If $\ell>k$, then $ex_{C_{2k+1}}(n,C_{2\ell+1})\leq 8(k+1)^2\ell n^{1+1/\lceil\ell/k\rceil}$.
    \end{itemize}
\end{theorem}

We also strengthen the first result in Theorem \ref{thm-4} when $G=C_4$ and $F$ is a cycle.

\begin{theorem} \label{Thm-c4}
\begin{itemize}
    \item [(1).] $\frac{\sqrt{2}n^{3/2}}{8}+o(n^{3/2})=ex(n/2,C_4)\leq {ex}_{C_4}(n,C_4)\leq \frac{\sqrt{2}n^{3/2}+n}{8}$.
    \item [(2).] For $k\geq 5$, we have
    $$
    ex_{C_k}(n,C_4)\leq \left\{
    \begin{array}{ll}
    \frac{\sqrt{3}n^{3/2}+n}{10}, &   k=5,\\
    \frac{\sqrt{6}n^{3/2}+n}{2k}, & k\geq 6.\\
    \end{array}
    \right.
    $$
\end{itemize}
\end{theorem}

\begin{remark}
According to Theorem 1.2 in \cite{Ma-Yang}, the gap between the lower and upper bounds of $ex_{C_4}(n,C_4)$ can be as large as $\Theta(n)$ for a positive density of integers $n$.
\end{remark}

\subsection{Three key points of multicolor Tur\'{a}n number}

In what follows we restrict our discussion to simple graphs.  
Imolay, Karl, Nagy and V\'{a}li \cite{Imolay} proved that $ex_F(n,G)=\Theta(n^2)$ if and only if there is no homomorphism from $G$ to $F$; for the nondegenerate case ($ex_F(n,G)=\Theta(n^2)$),  
\begin{align}\label{upper-lower}
\frac{n^2}{v(F)^2}+o(n^2)\leq ex_F(n,G)\leq \frac{ex(n,G)}{e(F)}+o(n^2).
\end{align} 
Li \cite{Li} extended these results to uniform hypergraphs.

It is worth mentioning that, as is well known, the Tur\'an exponent of any degenerate case is always less than a constant that strictly less than $2$.  
However, for the degenerate case of the multicolor Tur\'an number $ex_F(n,G)$, there exist graph pairs $(F,G)$ with $ex_F(n,G)=n^{2-o(1)}$ (for example $F=G=C_3$), whose exponent tends to $2$ as $n\to\infty$.  
This naturally brings three natural thresholds concerning $ex_F(n,G)$: the lower bound $n^2/v(F)^2+o(n^2)$ form Ineq.~(\ref{upper-lower}), the upper bound $ex(n,G)/e(F)+o(n^2)$ from Ineq.~(\ref{upper-lower}), and the $(6,3)$-type bound $n^{2-o(1)}$ (see Figure~\ref{fig-3-points}).  
An interesting question is to determine those graph pairs $(F,G)$ for which $ex_F(n,G)$ attains each of these points.

\begin{figure}[ht]
    \centering   \includegraphics[width=260pt]{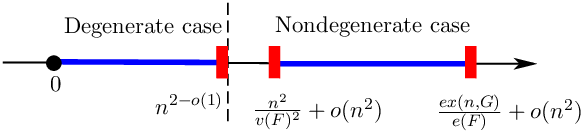}\\
    \caption{Three natural thresholds of multicolor Tur\'{a}n number.} \label{fig-3-points}
\end{figure}

Note that $\chi(G)>\chi(F)$ implies that there is no homomorphism from $G$ to $F$.
Imolay, Karl, Nagy and V\'{a}li \cite{Imolay} showed that $ex_F(n,G)$ attains its upper bound in Ineq.~(\ref{upper-lower}) when $\chi(G)>\chi(F)$, which encourages us to consider the asymptotic value of $ex_F(n,G)$ when $\chi(G)\le\chi(F)$. 
In fact, using the stability of extremal graphs for $ex(n,G)$, one can show that the condition $\chi(G)>\chi(F)$ is also necessary \cite{Li}: $ex_F(n,G)$ attains its upper bound in Ineq.~(\ref{upper-lower}) if and only if $\chi(G)>\chi(F)$.

Imolay, Karl, Nagy and V\'{a}li \cite{Imolay} posed the question of characterizing all graph pairs $(F,G)$ for which $ex_F(n,G)=n^2/v(F)^2+o(n^2)$. The first such pair was found by Kov\'{a}cs and Nagy \cite{Kovacs}: $F=C_5$ and $G=C_3$. Later, Balogh, Liebenau, Mattos and Morrison \cite{Balogh} refined this result by showing  $(n^2-10n)/25\leq ex_{C_5}(n,C_3)\leq (n^2+3n)/25+o(n)$.
 Characterizing when $ex_F(n,G)$ attains its lower bound in Ineq.~(\ref{upper-lower}) is a challenging problem. 
 Instead, we first study conditions that prevent it from attaining the lower bound in Ineq.~(\ref{upper-lower}).

\begin{theorem}\label{thm-1}
For graphs $G$ and $F$ with 
$\chi(G)\leq \chi(F)$, if the following conditions hold, then $ex_F(n,G)$ does not attain its lower bound in Ineq.~(\ref{upper-lower}).
\begin{enumerate}
    \item [(A).] There are two nonadjacent vertices $u,v$ in $F$ such that $u,v$ have at most one common neighbor;
    \item [(B).] there exists no homomorphism from $G$ to $F$, nor from $G$ to $F/\{u,v\}$.
\end{enumerate}
\end{theorem}

\begin{remark}
    Condition (A) is necessary: for $t\ge 2$ and the pair of graphs $F=C_5(t)$ and $G=C_3$, condition (B) holds for any two nonadjacent vertices $u,v$ coming from the same part of the blow-up, while condition (A) fails; nevertheless, $ex_{C_5(t)}(n,C_3)$ attains its lower bound in Ineq.~(\ref{upper-lower}) (see Theorem~\ref{thm-3+}). Here $G(t)$ denotes the $t$-blow-up of $G$.
\end{remark}

The following result provides a method to generate more such pairs from one that attains the lower bound.

\begin{theorem}\label{thm-2}
    Let $G,F$ be two graphs with $ex_F(n,G)=\frac{n^2}{v(F)^2}+o(n^2)$, and $s$ be a positive integer.
    \begin{enumerate}
        \item   [(1).]  If $F(s)$ can be decomposed into $s^2$ edge-disjoint copies of $F$, then $ex_{F(s)}(n,G)=\frac{n^2}{v(F(s))^2}+o(n^2)$.
        \item  [(2).]  Specifically, if $s$ is a prime with $v(F)\leq s$, then $ex_{F(s)}(n,G)=\frac{n^2}{v(F(s))^2}+o(n^2)$.
    \end{enumerate}  
\end{theorem}

For cycles $F$ and $G$, we extend the known results $F=C_5,G=C_3$ to more cycle pairs that attain the lower bound, and explore corresponding stabilities. 

\begin{theorem}\label{thm-3+}
    Suppose $k\geq 5$ is an odd integer, then 
    \begin{itemize}
      \item [(1).] for any positive integer $t$, $ex_{C_k(t)}(n,C_{k-2})=\frac{n^2}{(kt)^2}+o(n^2)$;
      \item [(2).] for any $\epsilon>0$, there exists an integer $N=N(k,t,\epsilon)$ such that for  each $n>N$, if $H$ is an $n$-vertex graph consisting of $ex_{C_k(t)}(n,C_{k-2})-\epsilon n^2$ monochromatic $C_k(t)$-copies and containing no rainbow $C_{k-2}$, then we can delete at most $2\epsilon(ktn)^2$ monochromatic $C_k(t)$-copies such that the resulting graph is a subgraph of a blow-up of $C_k$.
    \end{itemize}
\end{theorem}

Now, we investigate graph pairs $(F,G)$ with $ex_F(n,G)$ attains the $(6,3)$-type bound $n^{2-o(1)}$. Kov\'{a}cs and Nagy \cite{Kovacs} generalized the $(6,3)$-theorem by showing that $ex_F(n,G)=n^{2-o(1)}$ if $G$ contains a triangle and there exists a homomorphism from $G$ to $F$. We explore graph pairs where $G$ is triangle-free and $ex_F(n,G)=n^{2-o(1)}$. The \emph{odd girth} of $G$, denoted by $g_o(G)$, is defined as the length of a shortest odd cycle in $G$. 
The main result of this part is as follows.

\begin{theorem}\label{thm-3}
    If $g_o(F)=g_o(G)=k$ and there exist homomorphisms from both $F$ and $G$ to $C_k$, then $ex_{F}(n,G)=n^{2-o(1)}$.
    Specifically, $ex_{C_k}(n,C_k)=n^{2-o(1)}$.    
\end{theorem}

Based on the above theorems, we can completely determine the necessary and sufficient conditions for the multicolor Tur\'{a}n number to attain three natural thresholds when both $F$ and $G$ are cycles.

\begin{remark}
Recall that for simple graphs $G$ and $F$, $ex_F(n,G)$ attains the upper bound in Ineq.~(\ref{upper-lower}) if and only if $\chi(G)>\chi(F)$. This implies $ex_{C_k}(n,C_\ell)$ attains the upper bound in Ineq. (\ref{upper-lower}) if and only if $k$ is even and $\ell$ is odd. 
Theorems \ref{thm-1} and \ref{thm-3+} indicate that $ex_{C_k}(n,C_\ell)$ attains the lower bound in Ineq.(\ref{upper-lower}) if and only if $k,\ell$ are odd and $\ell=k-2$.
Theorems \ref{thm-4} and \ref{thm-3}  indicate that $ex_{C_k}(n,C_\ell)=n^{2-o(1)}$ if and only if $k=\ell\geq 3$ are odd integers.
\end{remark}

The paper is organized as follows.  
Section~2 collects several tools that will be used in subsequent proofs.  
In Section~3 we prove our first main result concerning the degenerate case (Theorem~\ref{thm-4}); the strengthened result when $F$ is a cycle and $G=C_4$ (Theorem~\ref{Thm-c4}) requires a more detailed discussion, and its proof is deferred to Section~7.  
Sections~4 and~5 focus on graph pairs that attain the lower bound in Ineq.~(\ref{upper-lower}); in these sections, we prove Theorems~\ref{thm-1} and~\ref{thm-2}, respectively.  
In Section~6 we present a proof of Theorem~\ref{thm-3}, which generalizes the $(6,3)$-theorem.
The final section contains some concluding remarks and some open problems.

\bigskip

\noindent{\bf Notations:} A {\em blow-up} of $G$ is a graph obtained by replacing each vertex $v\in V(G)$ with a vertex set $X_v$, and any $a\in X_u,b\in X_v$ form an edge only if $uv\in E(G)$. If $|X_v|=t$ for each $v\in V(G)$, then such a blow-up of $G$ is called a {\em $t$-blow-up} of $G$, and denoted by $G(t)$.
Each such $X_v$ is called a {\em part} of the blow-up.
For a vertex $v\in V(G)$ and a vertex set $A\subseteq V(G)$, the {\em distance} from $v$ to $A$, denoted by $dist_G(v,A)$, is the minimum length of a path with endpoints $v$ and some $a\in A$. 
For a hypergraph $\mathcal{H}$ and an edge $e\in E(\mathcal{H})$, we use $V(e)$ denote the set of vertices in $e$.

\section{Preliminaries}

In this section, we will list some tools that will be used in this paper, including necessary concepts, notations, and lemmas. In addition, we present the proof of Proposition \ref{prop-1}.

For two graphs $G$ and $H$, let ${G\choose H}$ denote the set of all copies of $H$ in $G$.
A map $\psi^*: {G\choose H}\rightarrow [0,1]$ such that for each $e\in E(G)$
$$\sum_{e\in E(H'), H'\in {G\choose H}}\psi^*(H')\leq 1$$
is called a {\em fractional $H$-packing} of $G$.
The maximum value of $\sum_{H'\in {G\choose H}}\psi^*(H')$, denoted by $\nu^*_G(H)$, is called the {\em fractional $H$-packing number} of $G$.
Recall that the $H$-packing number of $G$, denoted by $\nu_G(H)$, is the maximum number of edge-disjoint copies of $H$ in $G$.
It is clear that $\nu^*_G(H)\geq \nu_G(H)$.
The following result establishes a relationship between graph packing number and fractional packing.
\begin{theorem}[Haxell and R\"{o}dl~ \cite{Haxell}]\label{thm-Haxell}
    For any  fixed graph $H$ and an $n$-vertex graph $G$, $\nu^*_G(H)-\nu_G(H)=o(n^2)$.
\end{theorem}

The following lemma shows that in the multicolor Tur\'{a}n problem, there are very few non-rainbow copies of $G$.
\begin{lemma}\label{lem-star}
    For an $n$-vertex $r$-uniform hypergraph $H$ consisting of some monochromatic $F$-copies, there are $O(n^{v(G)-1})$ non-rainbow copies of $G$.
\end{lemma}
\begin{proof}
Since every non-monochromatic copy of $G$ contains two edges $e_1,e_2$ of the same color (otherwise it would be rainbow) and $|e_1\cup e_2|\geq r+1$, there are at most 
$$\#\mbox{monochromatic }F\mbox{-copies}\cdot{e(F)\choose 2}=O(n^r)$$
choices of such monochromatic edge pairs $e_1$ and $e_2$, and hence there are 
$$O(n^r)\cdot O(n^{v(G)-|e_1\cup e_2|})\leq O(n^r)\cdot O(n^{v(G)-(r+1)})=O(n^{v(G)-1})$$ 
non-rainbow copies of $G$.
\end{proof}

For each graph $G$, let 
$$\mathrm{Epi}(G)=\{G':\mbox{there is an epimorphism from }G\mbox{ to }G'\}.$$
The following two lemmas provide preprocessing methods for graphs, and we present their proofs simultaneously.

\begin{lemma}\label{pre-lem-2}
    Suppose $H$ is an $n$-vertex rainbow $G$-free graph consisting of monochromatic $F$-copies, then we can remove at most $o(n^2)$ monochromatic $F$-copies such that the resulting graph is $\mathrm{Epi}(G)$-free. 
\end{lemma}
\begin{proof}
We claim the following:
\begin{itemize}
\item[($\clubsuit$)] For each integer $1\leq \ell\leq v(G)$ and each $G'\in \mathrm{Epi}(G)$ with $v(G')=\ell$, the number of copies of $G'$ in $H$ is $o(n^{\ell})$.
\end{itemize}
Since $\mathrm{Epi}(G)$ is finite, by the Removal Lemma we can obtain an $\mathrm{Epi}(G)$-free graph by removing at most $o(n^2)$ edges from $H$, together with the corresponding monochromatic $F$-copies.

It suffices to prove statement ($\clubsuit$). By Lemma \ref{lem-star}, the number of copies of $G$ in $H$ is $O(n^{v(G)-1})$ (the case $\ell=v(G)$). Hence, assume $\ell<v(G)$. Suppose for contradiction that there exist a constant $c_0>0$ and infinitely many integers $n$ such that $H$ contains at least $c_0 n^{\ell}$ copies of some $G'\in\mathrm{Epi}(G)$ with $v(G')=\ell$. Regarding each vertex set of such a copy as an $\ell$-edge, we construct an $n$-vertex $\ell$-uniform hypergraph $\mathcal{G}$ without parallel edges.
Since there are finite many copies of $G'$ sharing the common vertex set, there exists a $c_0'$ with $0<c_0'\leq c_0$ such that $e(\mathcal{G})\geq c_0' n^{\ell}$. For each edge $e\in E(\mathcal{G})$, let $G'_e$ denote a corresponding copy of $G'$ on $V(e)$. For a subhypergraph $\mathcal{F}$ of $\mathcal{G}$, define $\mathcal{F}_H$ as the subgraph $\bigcup_{e\in E(\mathcal{F})} G'_e$ of $H$.

Let $K_{\ell;a}^{(\ell)}$ denote the complete $\ell$-partite $\ell$-uniform hypergraph with each part of size $a$. Let $\alpha$ be the independence number of $G$. Clearly, $G'(\alpha)$ contains a copy of $G$. Hypergraph Ramsey theory implies that there exists an integer $r$ such that any $\ell!$-coloring of the edges of $K_{\ell;r}^{(\ell)}$ yields a monochromatic $K_{\ell;\alpha}^{(\ell)}$. Since the Tur\'{a}n density of $K_{\ell;r}^{(\ell)}$ is zero, the Supersaturation Lemma guarantees that $\mathcal{G}$ contains $c_1 n^{r\ell}$ copies of $K_{\ell;r}^{(\ell)}$ for some $c_1>0$.

Fix a copy $G^*$ of $G'$. For each copy of $K_{\ell;r}^{(\ell)}$ in $\mathcal{G}$ with parts $X_1,\dots,X_\ell$, color each edge $e$ of this $K_{\ell;r}^{(\ell)}$ by the bijection $\psi_e: V(G^*)\to[\ell]$ satisfying: $uv\in E(G^*)$ if and only if there is an edge of $G'_e$ between $X_{\psi_e(u)}$ and $X_{\psi_e(v)}$. The number of possible colors is at most $\ell!$, so by the Ramsey property there exists a monochromatic copy $\mathcal{K}$ of $K_{\ell;\alpha}^{(\ell)}$ inside this $K_{\ell;r}^{(\ell)}$. Then $\mathcal{K}_H$ is isomorphic to $G'(\alpha)$, which contains a copy of $G$. Consequently, for every copy $\mathcal{L}$ of $K_{\ell;r}^{(\ell)}$ in $\mathcal{G}$, the subgraph $\mathcal{L}_H$ contains a copy of $G'(\alpha)$ and hence a copy of $G$.

Conversely, each copy of $G$ in $H$ can arise from at most $O(n^{r\ell-v(G)})$ copies $\mathcal{L}$ of $K_{\ell;r}^{(\ell)}$ in $\mathcal{G}$ (i.e., those $\mathcal{L}$ with $\mathcal{L}_H$ containing that $G$). Therefore,
$$
\#\{\text{copies of } G \text{ in } H\} \geq \frac{\#\{\text{copies of } K_{\ell;r}^{(\ell)} \text{ in } \mathcal{G}\}}{O(n^{r\ell-v(G)})} 
\geq \frac{c_1n^{r\ell}}{O(n^{r\ell-v(G)})}
=\Omega(n^{v(G)}),
$$
which contradicts the fact that the number of copies of $G$ in $H$ is $O(n^{v(G)-1})$.
\end{proof}

\begin{lemma}\label{lem-random}
If $F$ is a graph of order $k$ and $G$ is an $n$-vertex graph consisting of $t$ monochromatic $F$-copies $F_1,F_2,\ldots F_t$, then there is a partition $V_1,V_2,\ldots, V_k$ of $V(G)$ and a bijection $f:[k]\rightarrow V(F)$, such that there are at least $t/k^k$ $F_i$s satisfying  \begin{itemize}
    \item [(1).] $|V(F_i)\cap V_j|=1$ for each $j\in [k]$ (say $w_j=V(F_i)\cap V_j$), and 
    \item [(2).] $w_aw_b\in E(F_i)$ if and only if $f(a)f(b)\in E(F)$.
\end{itemize}
\end{lemma}
\begin{proof}
Partition $V(G)$ into $V_1,V_2,\ldots,V_k$ randomly and uniformly 
(i.e. put each vertex of $H$ into each $V_i$ with the probability $1/k$).
Let $\mathcal{S}=\{F_j:|V(F_j)\cap V_i|=1\mbox{ for each }i\in [k]\}$.
Since the expectation of $|\mathcal{S}|$ is $k!t/k^k$, there exists such a partition $V_1,V_2,\ldots,V_{k}$ of $V(G)$ such that $|\mathcal{S}|\geq k!t/k^k$.
Since there are $k!$ bijection from $[k]$ to $V(F)$, there exists a bijection $f$ such that there are at least $t/k^k$ $F_i$s satisfying statement (2).
\end{proof}

The following result, which appears in \cite{Kovacs}, is straightforward.

\begin{proposition}[Kov\'{a}cs and Nagy~\cite{Kovacs}]\label{leqleq}
If $G'$ is a subgraph of $G$ and $F'$ is a subgraph of $F$, then $ex_{F}(n,G')\leq ex_{F'}(n,G')\leq ex_{F'}(n,G)$.
\end{proposition}

A {\em theta graph}, denoted by $\theta_k$, is a cycle $C_k$ with a chord. 
Let $\theta_{\geq k}$ denote the set of all $\theta_t$ with $t\geq k$.
In the following, we give a result by F\"{u}redi and \"{O}zkahya (see Lemma 7 in \cite{Furedi}).

\begin{lemma}[F\"{u}redi and \"{O}zkahya~\cite{Furedi}]\label{Ozkahya}
Let $G$ be a $C_{2k+1}$-free graph and let $T$ be a subtree in $G$ with root $x$. 
Let $V_i$ denote the set of vertices of distance $i$ from $x$ in the tree $T$.  
Then for each $i\leq k$, $G[V_i]$ is $\theta_{\geq 2k}$-free.
\end{lemma}

At the end of this section, we add the proof of Proposition  \ref{prop-1}.

\noindent{\bf Proof of Proposition  \ref{prop-1}:}
It is clear that $\nu_n(F,G)\leq ex_F(n,G)$. It suffices to  show $\nu_n(F,G)+o(n^2)\geq ex_F(n,G)$.
If $H$ is an $n$-vertex rainbow $G$-free graph consisting of $ex_F(n,G)$ monochromatic $F$-copies, then by Lemma \ref{lem-star} and the hypergraph removal lemma, we may delete $o(n^r)$ edges together with the corresponding monochromatic $F$-copies to obtain a $G$-free graph that consists of $ex_F(n,G)-o(n^r)$ edge-disjoint copies of $F$. This yields $ex_F(n,G)\le \nu_n(F,G)+o(n^r)$.

\section{Proof of Theorem \ref{thm-4}}

\noindent{\bf Proof of Theorem \ref{thm-4}(1):}
 Let $v(F)=k$. It suffices to show that $ex_F(n,G)\leq k^k \cdot ex(n,G)$.
Suppose, to the contrary, that $ex_F(n,G)>k^k \cdot ex(n,G)$.
Let $H$ be an $n$-vertex rainbow $G$-free graph consisting of $ex_F(n,G)$ monochromatic $F$-copies.
By Lemma \ref{lem-random}, 
we can partition $V(H)$ into $V_1,V_2,\ldots,V_k$ so that there are at least $t=ex(n,G)+1$ monochromatic $F$-copies, say $F^1,F^2,\ldots,F^t$, satisfying two statements in Lemma \ref{lem-random}.
Let $i,j \in [k]$ be two indices  for which $e(V_i,V_j)\neq \emptyset$ and $|V_i \cup V_j|$ is maximized. Then $|V_i \cup V_j| \geq \delta(F)n/e(F)$.
Note that $H'=H[V_i\cup V_j]$ is a rainbow graph and 
$$e(H')\geq t>ex(n,G)\geq ex(|V_i\cup V_j|,G).$$ 
Hence, $H'$ contains a rainbow copy of $G$, contradicting the fact that $H$ is rainbow $G$-free.
Specifically, if $G=C_{2\ell}$, then $ex(n,G)=O(n^{1+1/\ell})$ (see \cite{ex-c2k}), and hence $ex_F(n,G)=O(n^{1+1/\ell})$.
\hfill$\square$ \medskip

Next, we prove Theorem \ref{thm-4}(2). The proof strategy is as follows. We first extract a high-minimum-degree rainbow subgraph, then analyze a lexicographically maximal layered rainbow tree, and finally count monochromatic cycles intersecting each level.

\smallskip

\noindent{\bf Proof of Theorem \ref{thm-4}(2):}
 Let $G$ be an $n$-vertex rainbow $C_{2\ell+1}$-free graph consisting of $ex_{C_{2k+1}}(n,C_{2\ell+1})$  monochromatic $C_{2k+1}$-copies.
Let $\delta$ denote the average number of monochromatic $C_{2k+1}$-copies containing a vertex. 
Then
\begin{align}\label{in-delta}
e(G)=(2k+1)ex_{C_{2k+1}}(n,C_{2\ell+1})=(2k+1)\delta n.
\end{align}
Delete all vertices $v$ with $d_G(v)<2\delta$ (i.e., vertices incident to fewer than $\delta$ monochromatic $C_{2k+1}$-copies), together with all monochromatic $C_{2k+1}$-copies incident to $v$.
In the resulting graph, if there exist vertices $w$ of degree less than $2\delta$, 
delete these vertices as well as all monochromatic $C_{2k+1}$-copies containing any of them.
Repeating this process, we eventually obtain a graph $G^*$.
Since $e(G^*)>e(G)-(2k+1)\delta n=0$, $G^*$ is a nonempty graph and has minimum degree at least $2\delta$.

For two integer sequences $\alpha=(a_1,a_2,\ldots,a_k)$ and $\beta=(b_1,b_2,\ldots,b_s)$, we say that $\alpha$ is lexicographically larger than $\beta$ if there exists $i\in[k]$ such that $a_i>b_i$ and $a_j=b_j$ for all $j<i$, or if $k>s$ and $a_i=b_i$ for each $i\in[s]$.
Choose a vertex $x$ in $G^*$ and let $T_x$ be a rainbow tree rooted at $x$ of height $\ell'\leq \ell$ (i.e., the maximum distance from $x$ to any vertex in $T_x$ is at most $\ell'$).
Then $V(T_x)$ can be partitioned into $\ell+1$ sets 
$\{x\},V_1,V_2,\ldots,V_{\ell'}$, where $V_i$ denotes the set of vertices at distance $i$ from $x$ in $T_x$.
Define $\gamma(T_x)=(|V_1|,|V_2|,\ldots,|V_{\ell'}|)$.
Among all rainbow trees in $G^*$ rooted at $x$ of height at most $\ell$, select one, denoted $T$, that maximizes $\gamma(T)$ with respect to the lexicographic order.
Assume that the height of $T$ is $h$.

Let $U_i$ denote the set of vertices at distance $i$ from $x$ in the rainbow tree $T$, and let $\mathcal{C}_{all}$ denote the set of monochromatic $C_{2k+1}$-copies that are incident to some vertex in  $\bigcup_{i\in[0,h]}U_i$.
We partition $\mathcal{C}_{all}$ into two subsets
$$\mathcal{T}=\{C\in \mathcal{C}_{all}:E(C)\cap E(T)\neq \emptyset\}\mbox{ and }\mathcal{C}=\{C\in \mathcal{C}_{all}:E(C)\cap E(T)= \emptyset\}.$$
Define 
$$
\mathcal{C}^*= \left\{
\begin{array}{ll}
\{C\in\mathcal{C}:V(C)\cap U_i\neq \emptyset\mbox{ for some }i\in[0,\ell-k]\}, &  h= \ell,\\
\{C\in\mathcal{C}:V(C)\cap U_i\neq \emptyset\mbox{ for some }i\in[0,h]\}, & h<\ell.\\
\end{array}
\right.
$$
“The following claims will be used repeatedly in the proof.

\begin{claim}\label{clm-odd-odd}
For each $C\in \mathcal{C}^*$, 
\begin{itemize}
    \item [(1).] every edge in $E(C)$ lies either in some $U_i$ or between consecutive parts $U_i$ and $U_{i+1}$, and 
    \item [(2).] $V(C)\subseteq V(T)$.
\end{itemize}
\end{claim}
\begin{proof}
We show the first statement. Suppose, to the contrary, that
there exists an edge $e=ab$ in $E(C)$ such that $a\in U_i$ and $b\in U_j$ for some $i,j\in[\ell]$ with $j-i\geq 2$.
Then $T\cup e$ contains a unique cycle $Q$, and $Q$ contains an edge $f$ that lies between $U_j$ and $U_{j-1}$.
Let $T'=T+e-f$.
Since the colors of $e$ and every edge in $T$ are distinct, $T'$ is a rainbow tree rooted at $x$ of height at most $\ell$.
However, since $j-i\geq 2$, we have $\gamma(T')>\gamma(T)$, a contradiction.

We now prove the second statement. Suppose, for contradiction, that $V(C)\setminus V(T)\neq\emptyset$.  
Let $\alpha$ be the smallest integer such that $V(C)\cap U_{\alpha}\neq\emptyset$.  
Define  
$$
Z_p := \{\, w : dist_C\bigl(w,\,V(C)\cap U_{\alpha}\bigr)=p \,\}.
$$  
Since $V(C)\setminus V(T)\neq\emptyset$, there exists a smallest integer $p\ge 1$ such that $Z_p\setminus V(T)\neq\emptyset$; choose $w'\in Z_p\setminus V(T)$.  
Then there is an edge $e=w'w''$ in $C$ with $w''\in Z_{p-1}$.  
By the first statement, we have $Z_{p-1}\subseteq U_{\alpha+p-1}$.  
Consequently, $T'' = T\cup\{e\}$ is a rainbow tree rooted at $x$ whose height is at most $\alpha+p$.  
It is easy to verify that $\gamma(T'') > \gamma(T)$.  

To obtain a contradiction, it suffices to show that $\alpha+p\le \ell$.  
Indeed, if $h<\ell$, then $\alpha+p-1\le h$, and hence $\alpha+p\le \ell$.  
If $h=\ell$, then since $C\in\mathcal{C}^*$ we have $\alpha\in[\ell-k]$, and the definition of $Z_p$ gives $p\le k$; hence again $\alpha+p\le\ell$.  
Thus the contradiction is established.
\end{proof}

\noindent{\bf Case 1:}  $h<\ell$.

By Claim \ref{clm-odd-odd}, we have that for each $C\in\mathcal{C}$, there exists an edge of $C$ that lies entirely in some $U_i$ with $i\in[h]$.  
Hence $|\mathcal{C}|\le \sum_{i\in[h]} e(G^*[U_i])$.  
Since $G^*$ is rainbow $C_{2\ell+1}$-free and $h<\ell$, Lemma~\ref{Ozkahya} implies that each $G^*[U_i]$ is $\theta_{\ge 2\ell}$-free.  
Since 
\begin{align}\label{ineq-turan-theta}
ex(N,\theta_{\ge 2\ell})\le 2(\ell-1)N
\end{align} 
(which is used in \cite{Furedi}), we obtain $e(G^*[U_i])\le 2(\ell-1)|U_i|$.  
Consequently,
$$
|\mathcal{C}|\le 2(\ell-1)\sum\nolimits_{i\in[h]}|U_i|
=2(\ell-1)\Bigl|\bigcup\nolimits_{i\in[h]}U_i\Bigr|.
$$

Note that $\mathcal{C}_{\text{all}} = \mathcal{T}\cup\mathcal{C}$ and $\mathcal{T}=|\bigcup_{i\in[h]}U_i|$. We have that
$$
\delta\Bigl|\bigcup\nolimits_{i\in[h]}U_i\Bigr|
\le (2k+1)|\mathcal{C}_{all}|
\le (2k+1)(|\mathcal{T}|+|\mathcal{C}|)
\le (2k+1)(2\ell-1)\Bigl|\bigcup\nolimits_{i\in[h]}U_i\Bigr|,
$$
which yields $\delta\le (2k+1)(2\ell-1)$.  Finally,
$$
ex_{C_{2k+1}}(n,C_{2\ell+1}) = \frac{e(G)}{2k+1} = \delta n
\le (2k+1)(2\ell-1)n
< 8(k+1)^2\ell\, n^{1+1/\lceil\ell/k\rceil}.
$$

\noindent{\bf Case 2:}  $h=\ell$.

The following special partitions of $\mathcal{T}$ and $\mathcal{C}^*$ are necessary.
Partition $\mathcal{T}$ into $\ell$ parts: for $i\in [\ell]$, let 
$$\mathcal{T}_i=\{C\in \mathcal{T}:  E(C)\cap E(T)\mbox{ is an edge between }U_{i-1}\mbox{ and }U_i\}.$$
Clearly, for each $i\in [\ell]$,
\begin{align}\label{ineq-odd-odd-1}
|\mathcal{T}_i|=|U_i|.
\end{align}
For each $C\in \mathcal{C}^*$, let $$top_C=\min\{i:U_i\cap V(C)\neq \emptyset\}.$$
Observe that, by the maximality of $\gamma(T)$, there are no monochromatic  $C_{2k+1}$-copies in $\mathcal{C}^*$ that intersect $x = U_0$.
Hence, we can further partition $\mathcal{C}^*$ into $\ell-k$ subsets: for each $i\in [\ell-k]$, let 
$$\mathcal{C}_i=\{C\in \mathcal{C}^*:top_C=i\}.$$

\begin{claim}\label{clm-odd-odd-1}
    Assume that $C$ is a monochromatic $C_{2k+1}$-copy containing some vertex in $U_i$.
    If $C\in\mathcal{T}$, then $C\in \bigcup_{j\in [0,i+k]}\mathcal{T}_j$; if $C\in\mathcal{C}^*$, then $C\in \bigcup_{j\in [\max\{0,i-k\},i]}\mathcal{C}_j$.
\end{claim}
\begin{proof}
    If $C\in\mathcal{C}^*$, then the result follows immediately from statement (1) in Claim \ref{clm-odd-odd}.
    Hence, assume that $C\in\mathcal{T}$.
    Without loss of generality, let $p$ be the integer such that $C \in \mathcal{T}_p$.
    By the definition of $\mathcal{T}_p$, there exists an edge of $C$, say $e=ab$, that also belongs to $T$, with $a\in U_{p-1}$ and $b\in U_p$.
    
    Suppose for the contrary that  $p\notin [0,i+k]$.
    Then $p-i>k$. 
    Since $C$ contains a vertex in $U_i$ (say $z$) and $dist_C(z,b)\leq k$, there exists an integer $q\geq i$ and an integer $\alpha\geq 2$ such that $q+\alpha\leq p$ and $C$ contains an edge $f=a'b'$ with $a'\in U_q$ and $b'\in U_{q+\alpha}$. 
    Clearly, 
    $$T'=T\left[\bigcup\nolimits_{i\in[0, q+1]}U_j\right]\cup\{f\}$$ 
    is a new rainbow tree rooted at $x$ of height at most $\ell$, and $\gamma(T')>\gamma(T)$, a contradiction.
\end{proof}

Since each $C\in \mathcal{C}^*$ is an odd cycle, by statement (1) in Claim \ref{clm-odd-odd}, there exists an $i\in[\ell]$ and an edge of $C$, denoted by $e_C$, such that $e_C\in G^*[U_i]$.
Let 
$$R=\{e_C:C\in \mathcal{C}^*\}\mbox{ and }R_i=R\cap G^*[U_i].$$
Clearly, each $E(T)\cup R_i$ is rainbow.
Since $G^*$ is rainbow $C_{2\ell+1}$-free, by Lemma \ref{Ozkahya}, $G^*[R_i]$ is $\theta_{\geq 2\ell}$-free.
By Ineq. (\ref{ineq-turan-theta}),
$$|R_i|=e(G^*[R_i])\leq 2(\ell-1)|U_i|.$$
Let $I_p=[\max\{0,p-k\},p+k]$.
It is obvious that for each $i\in [\ell-k]$, 
$I_p\subseteq [0,\ell]$.

\begin{claim}\label{clm-odd-odd-2}
For each $ i\leq \ell-k$, 
$\sum_{j\in[\max\{0,i-k\},i]}|\mathcal{C}_j|\leq 2(\ell-1)\sum_{j\in I_i}|U_j|$.
\end{claim}
\begin{proof}
  Given an $i\in[\ell-k]$, statement (1) in Claim \ref{clm-odd-odd} implies that for each $C\in \mathcal{C}_i$, $V(C)\subseteq \bigcup_{i\leq j\leq i+k}U_j$.
  Hence, $e_C$ lies in some $U_j$, $i\leq j\leq i+k$.
  Consequently, 
  $$\sum_{j\in[\max\{0,i-k\},i]}|\mathcal{C}_j|
  \leq \sum_{j\in I_i}|R_j|\leq 2(\ell-1)\sum_{j\in I_i}|U_j|.$$
\end{proof}

Note that for each monochromatic $C_{2k+1}$-copy $C$, if $V(C)\cap U_i\neq \emptyset$, then $C$ {\em contributes} at most $$\omega_i(C)\leq 2e(C)=4k+2$$ 
to the computation of $\sum_{v\in U_i}d_{G^*}(v)$.
By Claim \ref{clm-odd-odd-1}, $\mathcal{T}$ contributes 
\begin{align}\label{ineq-contribute-1}
\sum_{j\in [0,i+k]}\sum_{C\in\mathcal{T}_j}\omega_i(C)\leq (4k+2)\sum_{j\in [0,i+k]}|\mathcal{T}_j|
\overset{\text{by Ineq. (\ref{ineq-odd-odd-1})}}{=} (4k+2)\sum_{j\in [0,i+k]}|U_j|
\end{align}
to $\sum_{v\in U_i}d_{G^*}(v)$, and $\mathcal{C}^*$ contributes
\begin{align}\label{ineq-contribute-2}
\sum_{j\in [\max\{0,i-k\},i]}\sum_{C\in\mathcal{C}_j}\omega_i(C)
\leq (4k+2)\sum_{j\in [\max\{0,i-k\},i]}|\mathcal{C}_j|\leq 2(\ell-1)(4k+2)\sum_{j\in I_i}|U_j|
\end{align}
to $\sum_{v\in U_i}d_{G^*}(v)$ (the last inequality holds due to Claim \ref{clm-odd-odd-2}).
Since each vertex in $U_i$, $i\in[0,\ell-k]$,  is incident only to monochromatic $C_{2k+1}$-copies belonging to $\mathcal{T}$ or $\mathcal{C}^*$, it follows that
\begin{align}\label{ineq-sum}
\sum_{v\in U_i}d_{G^*}(v)
&\leq (4k+2)\sum_{j\in [0,i+k]}|U_j|+ (4k+2)(2\ell-2)\sum_{j\in I_i}|U_j|.
\end{align}
Since each $i$ is contained  in at most $2k+1$ intervals $I_j$, 
it follows that for each $t\in[\ell-k]$, 
\begin{align}\label{ineq-sum-1}
\sum_{i\in[0,t]}\sum_{j\in I_i}|U_j|\leq (2k+1)\sum_{j\in[0,t+k]}|U_j|.
\end{align}
Similarly, 
\begin{align}\label{ineq-sum-2}
\sum_{i\in[0,t]}\sum_{j\in [0,i+k]}|U_j|\leq (t+1) \sum_{j\in[0,t+k]}|U_j|\leq (\ell-k+1) \sum_{j\in[0, t+k]}|U_j|.
\end{align}
Consequently, by Ineqs. (\ref{ineq-contribute-1}) to (\ref{ineq-sum-2}),
\begin{align*}
\sum_{i\in[0,t]}\sum_{v\in U_i}d_{G^*}(v)  
&\leq (4k+2)\sum_{i\in[0,t]}\sum_{j\in [0,i+k]}|U_j|+(4k+2)(2\ell-2)\sum_{i\in[0,t]}\sum_{j\in I_i}|U_j|\\
&\leq \left[(4k+2)(\ell-k+1)+2(2k+1)^2(2\ell-2)\right]\sum_{i\in[0,t+k]}|U_i|\\
&<16(k+1)^2\ell\sum_{i\in[0,t+k]}|U_i|.
\end{align*}
Recall that the minimum degree of $G^*$ is at least $2\delta$.
Hence,
$$2\delta\sum_{i\in[0,t]}|U_i|\leq \sum_{i\in[0,t]}\sum_{v\in U_i}d_{G^*}(v)< 16(k+1)^2\ell\sum_{i\in[0, t+k]}|U_i|.$$
Let $W_i=\bigcup_{j\in[0,i]}U_j$ and $\lambda =[8(k+1)^2\ell]^{-1}$.
Then 
$$\delta|W_t|=2\delta\sum_{i\in[0,t]}|U_i|\leq 8(k+1)^2\ell\sum_{i\in[0,t+k]}|U_i|=8(k+1)^2\ell |W_{t+k}|=|W_{t+k}|/\lambda.$$
By the maximality of $\gamma(T)$, $|W_1|=|U_1|+1>\delta$.
Let $\ell=pk+q$, where $q\in[0,k-1]$.
Then 
$$n\geq |W_{\ell}|\geq \lambda\delta |W_{\ell-k}|\geq \ldots\geq (\lambda\delta)^p|W_q|.$$
If $q=0$, then $n\geq (\lambda\delta)^p$; if $q>0$, then $n\geq (\lambda\delta)^p|W_q|\geq (\lambda\delta)^p|W_1|\geq \lambda^p\delta^{p+1}$.
Therefore, 
$$n\geq \lambda^{\lfloor \ell/k\rfloor}\delta^{\lceil\ell/k\rceil}.$$
Consequently, 
$$ex_{C_{2k+1}}(n,C_{2\ell+1})\leq \delta n\leq \frac{n^{1+1/\lceil\ell/k\rceil}}{\lambda}=[8(k+1)^2\ell]n^{1+1/\lceil\ell/k\rceil},$$
the proof thus completes.
\hfill$\square$ \medskip

\section{Proofs of Theorems \ref{thm-1} and \ref{thm-2}}

In this section, we will prove Theorems \ref{thm-1} and \ref{thm-2}.

\noindent\textbf{Proof of Theorem \ref{thm-1}:}
Since there is no homomorphism from $G$ to $F$, we have $\chi(G)\geq 3$.
Therefore, $3\leq \chi(G)\leq \chi(F)$.

Let $v(F)=f$ and $V(F)=\{v_1,v_2,\ldots,v_{f-2},u,v\}$.
Assume that $F' = F/\{u,v\}$ is obtained by shrinking $u$ and $v$ to a single vertex $v_0$.
Then $V(F') = \{v_1,v_2,\ldots,v_{f-2},v_0\}$.
If $u$ and $v$ have a common neighbor, we may assume it is $v_1$.

\begin{claim}\label{lem1-clm-1}
$e(F'-\{v_0,v_1\})\geq 1$.
\end{claim}
\begin{proof}
Suppose for contradiction that $e(F'-\{v_0,v_1\})=0$.
Then $\{v_2,v_3,\ldots,v_{f-2}\}$, $\{u,v\}$ and $\{v_1\}$ are three independent sets in $F$.
Together with $3\leq \chi(G)\leq \chi(F)$, we obtain $\chi(F)=\chi(G)=3$.
Moreover, $v_1$ is the unique common neighbor of $u$ and $v$.
Consequently, for each $i$ with $2\leq i\leq f-2$, $v_i$ is adjacent to at most one of $u$ and $v$.
Since $\chi(F)=3$, there exists an integer $i$ such that either $v_iv_1, v_iu\in E(F)$ or $v_iv_1, v_iv\in E(F)$, which implies that $F$ contains a triangle.
However, $\chi(G)=3$ implies the existence of a homomorphism from $G$ to $F$, a contradiction.
\end{proof}

For sufficiently large $n$ and an $n$-vertex set $V$, let $\mathcal{P}=\{V_0,V_1,V_2,\ldots,V_{f-2}\}$ be a partition of $V$ such that  
$$
|V_2|=|V_3|=\ldots=|V_{f-2}|=\left\lfloor\frac{n-2}{2\sqrt{2}+f-3}\right\rfloor
$$
and $-1\leq |V_0|-|V_1|\leq 1$.
Clearly,
$$
|V_0|,|V_1|\geq\left\lfloor\frac{\sqrt{2}(n-2)}{2\sqrt{2}+f-3}+1\right\rfloor\geq \frac{\sqrt{2}(n-2)}{2\sqrt{2}+f-3}.
$$
Let $H$ be an $n$-vertex graph on $V$ with edge set
$$
E(H)=\bigcup_{v_iv_j\in E(F')}\{xy : x\in V_i,\ y\in V_j\}.
$$
Since there is no homomorphism from $G$ to $F'$, the graph $H$ is $G$-free.
By Claim \ref{lem1-clm-1}, we can choose $V_p,V_q\in\mathcal{P}$ such that $v_pv_q\in E(F'-\{v_0,v_1\})$.
It is clear that
$$
e[V_p,V_q]=|V_p||V_q|=\left\lfloor\frac{n-2}{2\sqrt{2}+f-3}\right\rfloor^2 = \min\{e[V_i,V_j] : v_iv_j\in E(F')\}.
$$
Choose a positive real number
$$
\epsilon < \frac{3-2\sqrt{2}}{2\sqrt{2}-3+f} < \frac{1.8}{f-1.8}.
$$
Then $1/(2\sqrt{2}+f-3) > (1+\epsilon)/f$, and hence
$$
e[V_p,V_q] > \frac{(1+\epsilon)^2 n^2}{f^2}.
$$

Let $\mathcal{A}$ be the set of all copies $F_0$ of $F$ in $H$ such that $|V(F_0)\cap V_i|=1$ for each $i\in[f-2]$ and $|V(F_0)\cap V_0|=2$.
For each edge $e\in E(H)$, let $\mathcal{A}_e$ denote the set of copies of $F$ in $\mathcal{A}$ that contain $e$.
By symmetry, for any two edges $e,e'$ connecting the same parts in $\mathcal{P}$, we have $|\mathcal{A}_e| = |\mathcal{A}_{e'}|$.
Therefore, for each $v_iv_j\in E(F')$ and each $e\in E[V_i,V_j]$,
\begin{itemize}
\item if $\{i,j\}\neq\{0,1\}$ or $\{i,j\}=\{0,1\}$ but $v_1$ is not the common neighbor of $u$ and $v$, then each $F_0\in\mathcal{A}$ contains exactly one edge in $E[V_i,V_j]$, so $|\mathcal{A}_e| = |\mathcal{A}| / e[V_i,V_j]$;
\item if $\{i,j\}=\{0,1\}$ and $v_1$ is the common neighbor of $u$ and $v$, then each $F_0\in\mathcal{A}$ contains exactly two edges in $E[V_i,V_j]$, so $|\mathcal{A}_e| = 2|\mathcal{A}| / e[V_i,V_j]$.
\end{itemize}

For each $F_0\notin \mathcal{A}$, define $\psi^*(F_0)=0$; for each $F_0\in \mathcal{A}$, choose an edge $e^*\in E[V_p,V_q]$ (note that $\{p,q\}\neq\{0,1\}$) and define $\psi^*(F_0)=1/|\mathcal{A}_{e^*}|$, which implies $\psi^*(F_0)=e[V_p,V_q]/|\mathcal{A}|$.
Now observe that
\begin{itemize}
\item if $\{i,j\}\neq\{0,1\}$ or $\{i,j\}=\{0,1\}$ but $v_1$ is not the common neighbor of $u$ and $v$, then
$$
\sum_{F_0\in \mathcal{A}_e}\psi^*(F_0)=|\mathcal{A}_e|\cdot\frac{e[V_p,V_q]}{|\mathcal{A}|}=\frac{e[V_p,V_q]}{e[V_i,V_j]}\leq 1;
$$
\item if $\{i,j\}=\{0,1\}$ and $v_1$ is the common neighbor of $u$ and $v$, then
$$
\sum_{F_0\in \mathcal{A}_e}\psi^*(F_0)=|\mathcal{A}_e|\cdot\frac{e[V_p,V_q]}{|\mathcal{A}|}=\frac{2e[V_p,V_q]}{e[V_i,V_j]}\leq 2\left\lfloor\frac{n-2}{2\sqrt{2}+f-3}\right\rfloor^{2}\cdot\left(\frac{\sqrt{2}(n-2)}{2\sqrt{2}+f-3}\right)^{-2}\leq 1.
$$
\end{itemize}
Thus $\psi^*$ is a fractional $F$-packing of $H$, and
$$
\nu^*_H(F)\geq \sum_{F_0\in\mathcal{A}}\psi^*(F_0)=\sum_{e\in E[V_p,V_q]}\sum_{F_0\in \mathcal{A}_e} \frac{e[V_p,V_q]}{|\mathcal{A}|}=e[V_p,V_q]> \frac{(1+\epsilon)^2 n^2}{f^2}.
$$
By Theorem \ref{thm-Haxell}, we obtain
$$
\nu_H(F)\geq \nu^*_H(F)+o(n^2)\geq \frac{(1+\epsilon)^2 n^2}{f^2}+o(n^2).
$$
Consequently, since $H$ is $G$-free,
$$
ex_F(n,G)\geq \nu_H(F)\geq \frac{(1+\epsilon)^2 n^2}{f^2}+o(n^2).
$$
\hfill$\square$ \medskip

\noindent{\bf Proof of Theorem \ref{thm-2}:}
    We prove the first statement. Suppose for contradiction that there exists $\epsilon>0$ such that for infinitely many $n$,
$$
ex_{F(s)}(n,G) \geq \frac{n^2}{v(F(s))^2} + \epsilon n^2.
$$
For each such $n$, let $H_n$ be a rainbow $G$-free graph consisting of $ex_{F(s)}(n,G)$ monochromatic $F(s)$-copies.
By Lemma \ref{lem-star} and the graph Removal Lemma, we may delete at most $o(n^2)$ monochromatic $F(s)$-copies from $H_n$ to obtain a $G$-free graph $H_n'$.
Clearly, $H_n'$ contains at least
$$
\frac{n^2}{v(F(s))^2} + \frac{\epsilon n^2}{2}
$$
monochromatic copies of $F(s)$.
Since each monochromatic copy of $F(s)$ can be decomposed into $s^2$ edge‑disjoint copies of $F$, the graph $H_n'$ can be decomposed into
$$
s^2 \cdot ex_{F(s)}(n,G) \geq \frac{n^2}{v(F)^2} + \frac{\epsilon s^2 n^2}{2}
$$
edge‑disjoint copies of $F$ (each regarded as a monochromatic $F$-copy).
Because $H_n'$ is $G$-free, we obtain
$$
ex_{F}(n,G) \geq \frac{n^2}{v(F)^2} + \frac{\epsilon s^2 n^2}{2}
$$
for infinitely many $n$, contradicting the fact that $ex_{F}(n,G) = \frac{n^2}{v(F)^2} + o(n^2)$.

We now show the second statement. For any prime $s$ satisfying $v(F)\le s$, it is enough to show that $F(s)$ can be decomposed into $s^2$ edge‑disjoint copies of $F$.
Let $v(F)=t$ and let $R$ be the $s$-blow‑up of $K_t$ with parts $V_1=V_2=\cdots=V_t=\{0,1,\ldots,s-1\}$.
For any $x\in V_1$ and $y\in V_2$, define $K^{x,y}$ to be the copy of $K_t$ induced by the vertices $x$, $y$, and $x+iy$  (taking modulo $s$). in $V_{i+2}$ for each $i\in[t-2]$.
    \begin{claim}\label{clm-decompo}
        For two distinct pairs $(x,y)$ and $(x',y')$, where $x,x'\in V_1$ and $y,y'\in V_2$, we have that $K^{x,y}$ and $K^{x',y'}$ are edge-disjoint.
    \end{claim}
    \begin{proof}
    Suppose for the contrary that $e\in E(K^{x,y})\cap E(K^{x',y'})$.
    First, assume that $e$ is an edge between $V_{i+2},V_{j+2}$ for some $t-2\geq i>j\geq 1$. Then $e=\{x+iy,x+jy\}=\{x'+iy',x'+jy'\}$.
    Clearly, $s$ divides both $(x+iy)-(x'+iy')$ and $(x+jy)-(x'+jy')$. Subtracting, we obtain $s\mid (i-j)(y-y')$.
    Since $s$ is a prime, either $s\mid(i-j)$ or $s\mid(y-y')$.
    Since $0<i-j<t\le s$, $s$ cannot divide $i-j$.
    As $y,y'\in[s]$, we must have $y=y'$.
    Now $s\mid [(x+iy)-(x'+iy')]$ together with $y=y'$ gives $s\mid(x-x')$. Thus $x=x'$, contradicting the assumption that $(x,y)\neq(x',y')$.
    
    Next, assume that either $e=\{x,x+iy\}=\{x',x'+iy'\}$ or $e=\{y,x+iy\}=\{y',x'+iy'\}$ for some $i\in[t-2]$ (by symmetry, say $e=\{x,x+iy\}=\{x',x'+iy'\}$).
    It is clear that $s\mid(x-x')$ and $s\mid i(y-y')$.
    Since $s$ is a prime, $0<i\leq s-2$ and $x,x',y,y'\in[s]$, it follows that $x=x'$ and $y=y'$, a contradiction.
\end{proof}

By Claim \ref{clm-decompo}, $R$ can be decomposed into $s^2$ edge-disjoint copies of $K_t$.
Note that 
$$R-\bigcup_{ij\notin E(F)}E(V_i,V_j)$$ is an $F(s)$ and 
$$F^{x,y}=K^{x,y}-\bigcup_{ij\notin E(F)}E(V_i,V_j)$$ 
is a copy of $F$ for each $x\in V_1,y\in V_2$.
Hence, $\{F^{x,y}:x\in V_1,y\in V_2\}$ is an $F$-decomposition of $F(s)$.
\hfill$\square$ \medskip

\section{Proof of Theorem \ref{thm-3+}}

\noindent{\bf Proof of Theorem \ref{thm-3+}(1):} We first show that $ex_{C_k}(n, C_{k-2}) = n^2/k^2 + o(n^2)$.
Let $G$ be a rainbow $C_{k-2}$-free graph consisting of $ex_{C_k}(n, C_{k-2})$ monochromatic copies of $C_k$.
Since $\{ C_\ell : 3 \le \ell \le k-2 \text{ and } \ell \text{ is odd} \} \subseteq \mathrm{Epi}(C_k)$,
by Lemma \ref{pre-lem-2} we may obtain an $n$-vertex graph $G'$ consisting of $ex_{C_k}(n, C_{k-2}) - o(n^2)$ monochromatic $C_k$-copies such that $G'$ is $C_\ell$-free for every odd $\ell$ with $3 \le \ell \le k-2$.
Denote by $\mathcal{F}$ the set of monochromatic $C_k$-copies in $G'$.

For each vertex $v \in V(G')$, the vertex $v$ is contained in exactly $d_{G'}(v)/2$ members of $\mathcal{F}$.
Therefore,
\begin{align}\label{Ineq-1-1}
\sum_{F \in \mathcal{F}} \sum_{v \in V(F)} d_{G'}(v)
= \sum_{v \in V(G')} \frac{d_{G'}(v)^2}{2}
\ge \frac{2 e(G')^2}{n}
= \frac{2 (k |\mathcal{F}|)^2}{n}.
\end{align}
On the other hand, for each monochromatic $k$-cycle $C$ of $G'$, we have that $|N_{G'}(v) \cap V(F)| \le 2$ for each $v \in V(G')$; otherwise, an odd cycle $C_\ell$ with $3 \le \ell \le k-2$ would appear, a contradiction.
Consequently, for each $F \in \mathcal{F}$,
$$\sum_{v \in V(F)} d_{G'}(v) \le 2 v(G') = 2n.$$
Thus,
$$\sum_{F \in \mathcal{F}} \sum_{v \in V(F)} d_{G'}(v) \le 2n |\mathcal{F}|.$$
Combining this with inequality (\ref{Ineq-1-1}) yields $|\mathcal{F}| \le n^2 / k^2$.
Hence,
$$ex_{C_k}(n, C_{k-2}) = \frac{n^2}{k^2} + o(n^2).$$

We now show that $ex_{C_k(t)}(n, C_{k-2}) = n^2/(kt)^2 + o(n^2)$. By the first statement of Theorem \ref{thm-2}, it suffices to prove that $H = C_k(t)$ can be decomposed into $t^2$ edge-disjoint copies of $C_k$.
Let $V_i = \{0,1,\ldots,t-1\}$, $i\in[k]$, be the $k$ parts of $H$. Edges of $H$ appear only in $E[V_1, V_k]$ or $E[V_i, V_{i+1}]$ for $i \in [k-1]$. For each $y \in V_1$ and $x \in V_2$, define $C^{y,x}$ as the $k$-cycle with vertex set ${y, x, x+y, x+2y, \dots, x+(k-2)y, y}$, where  $x + iy \in V_{i+2}$ for each $i \in [k-2]$ (taking
modulo $t$).
It remains to show that for distinct edges $xy$ and $x'y'$ with $y, y' \in V_1$ and $x, x' \in V_2$, the cycles $C^{y,x}$ and $C^{y',x'}$ are edge-disjoint. Suppose, to the contrary, that they share a common edge $e$.
If $e = \{x + iy, x + (i+1)y\} = \{x' + iy', x' + (i+1)y'\}$ for some $i \in [k-3]$, then we must have $x = x'$ and $y = y'$, a contradiction. Similarly, the remaining two cases, namely $e = \{y, x + (k-2)y\} = \{y', x' + (k-2)y'\}$ and $e = \{x, x + y\} = \{x', x' + y'\}$, both lead to $x = x'$ and $y = y'$, again a contradiction. Therefore, $C^{y,x}$ and $C^{y',x'}$ are edge-disjoint.
\hfill$\square$ \medskip

\smallskip

\noindent{\bf Proof of Theorem \ref{thm-3+}(2):} 
Assume that $H$ is an $n$-vertex graph consisting of $ex_{C_k(t)}(n, C_{k-2})-\epsilon n^2$ edge-disjoint monochromatic $C_k(t)$-copies and containing no rainbow $C_{k-2}$.
Since every odd cycle $C_\ell$ with $3\le \ell\le k-2$ belongs to $Epi(C_{k-2})$, Lemma \ref{pre-lem-2} implies that, after deleting at most $\epsilon n^2/2$ monochromatic $C_k(t)$-copies, we may assume that the remaining graph, denoted by $H_0$, is $C_\ell$-free for all odd $\ell\in[3,k-2]$.
Moreover, for each monochromatic $C_k(t)$-copy $F$ (with parts $X_0,X_1,\ldots,X_{k-1}$), the following holds.

\begin{claim}\label{clm-stab-neighbors}
\begin{itemize}
  \item[(i)] For every vertex $v\in V(H_0)$, $|N_{H_0}(v)\cap V(F)|\le 2t$.
  \item[(ii)] If $|N_{H_0}(v)\cap V(F)| = 2t$, then $N_{H_0}(v)\cap V(F)=X_{i-1}\cup X_{i+1}$ for some $i\in\{0,1,\ldots,k-1\}$ (indices taken modulo $k$).
\end{itemize}
\end{claim}
\begin{proof}
  Suppose, without loss of generality, that $|N_{H_0}(v)\cap V(F)|>t$.
  Then there exist two indices $a,b\in\{0,1,\ldots,k-1\}$ such that $N_{H_0}(v)\cap X_a\neq\emptyset$ and $N_{H_0}(v)\cap X_b\neq\emptyset$.
  It must be that $\{a,b\}=\{i-1,i+1\}$ for some $i$; otherwise $H_0$ would contain an odd cycle $C_\ell$ with $3\le\ell\le k-2$, a contradiction.
  Consequently, no third index $c$ satisfies $N_{H_0}(v)\cap X_c\neq\emptyset$, and the claim follows.
\end{proof}

Let $\mathcal{C}$ denote the family of monochromatic $C_k(t)$-copies in $H_0$.
Then $|\mathcal{C}|\geq n^2/(kt)^2 - 3\epsilon n^2/2$.
Since $C_k(t)$ is $2t$-regular, each vertex $v$ belongs to exactly $d_{H_0}(v)/(2t)$ members of $\mathcal{C}$.
Using Lemma \ref{clm-stab-neighbors}, we obtain
$$
\sum_{F\in\mathcal{C}}\sum_{v\in V(F)} d_{H_0}(v)
\;\ge\; \sum_{v\in V(H_0)} \frac{d_{H_0}^2(v)}{2t}
\;\ge\; \frac{2e(H_0)^2}{tn}
\;\ge\; \frac{2\bigl(kt^2|\mathcal{C}|\bigr)^2}{tn}
\;=\;2tn|\mathcal{C}| - 3k^2t^3\epsilon n|\mathcal{C}|.
$$
Thus there exists a monochromatic $C_k(t)$-copy $C_0$ (with parts $Y_0,Y_1,\ldots,Y_{k-1}$) such that the set
$$
B = \bigl\{ u\in V(G) : |N_{H_0}(u)\cap V(C_0)| < 2t \bigr\}
$$
has size at most $3k^2t^3\epsilon n$.

Every vertex outside $B$ has exactly $2t$ neighbours on $C_0$. By part (ii) of the claim, if these neighbours lie in $Y_a\cup Y_b$ then $a\equiv b\pm1\pmod{k}$.
Therefore the vertices outside $B$ are partitioned into $k$ classes:
$$
V_i = \bigl\{ u\in V(G)\setminus B : N_{H_0}(u)\cap V(C_0) = Y_{i-1}\cup Y_{i+1}\bigr\},\qquad i\in\{0,1,2,\ldots,k-1\},
$$
In particular, $Y_i\subseteq V_i$.

\begin{lemma}\label{clm-stability}
If $u\in V_i$, $v\in V_j$ and $uv\in E(G)$, then $j\equiv i\pm1\pmod{k}$.
Consequently, $G[V(G)\setminus B] \subseteq C_k(V_0,V_1,\dots,V_{k-1})$, the blow‑up of $C_k$ with parts $V_0,\dots,V_{k-1}$.
\end{lemma}
\begin{proof}
Choose a $k$-cycle $x_0x_1\ldots x_{k-1}x_0$ in $C_0$ with $x_i\in Y_i$.
If $i=j$, then $uvx_{i+1}u$ forms a triangle in $H_0$, contradicting $H_0$ contains no odd cycle of length less than $k$.
If $j\equiv i+s\pmod{k}$ for some $2\le s\le k-2$, then either $ux_{i+1}x_{i+2}\ldots x_{j-1}vu$ or $ux_{i-1}x_{i-2}\ldots x_{j+1}vu$ is an odd cycle of length less than $k$, again a contradiction.
Hence $j\equiv i\pm1\pmod{k}$.
\end{proof}

Now delete every member of $\mathcal{C}$ that intersects $B$. Since a vertex $v$ belongs to $d_{H_0}(v)/(2t)$ cycles of $\mathcal{C}$,
$$
\bigl|\bigl\{ C'\in\mathcal{C} : V(C')\cap B\neq\emptyset \bigr\}\bigr|
\;\le\; \frac1{2t}\sum_{v\in B} d_{H_0}(v)
\;\le\; \frac1{2t}\,|B|\,n
\;=\; 3(kt)^2\epsilon n^2/2.
$$
Let $\mathcal{C}_1$ be the remaining family. Then
$$
|\mathcal{C}_1| = |\mathcal{C}| - 3(kt)^2\epsilon n^2/2\geq \frac{n^2}{(kt)^2}-2(kt)^2\epsilon n^2,
$$
and the union of its members is contained in $C_k(V_0,\dots,V_{k-1})$.
\hfill$\square$ \medskip

\section{Proof of Theorem \ref{thm-3}}

Since there exists a homomorphism from $F$ to $C_k$, it follows that $F$ is a subgraph of $C_k(t)$ for some positive integer $t$. Note that both $F$ and $G$ contain $C_k$ as a subgraph. By Proposition~\ref{leqleq},
$$
ex_{C_k(t)}(n, C_k) \le ex_{C_k(t)}(n, G) \le ex_F(n, G).
$$
Since there exists a homomorphism from $G$ to $C_k$ and $F$ contains a copy of $C_k$, there exists a homomorphism from $G$ to $F$; consequently, $ex_F(n, G) = o(n^2)$. Thus it suffices to show that $ex_{C_k(t)}(n, C_k) \ge n^{2-o(1)}$.
The proof strategy is as follows: first, we verify that $ex_{C_k}(n, C_k) \ge n^{2-o(1)}$, and then, via a blow-up construction, we prove that $ex_{C_k(t)}(n, C_k) \ge n^{2-o(1)}$.

We first establish $ex_{C_k}(n, C_k) \ge n^{2-o(1)}$.
The proof relies on a known result of Ruzsa (see Theorem 2.3 in \cite{Ruzsa}), which asserts the existence of a subset $A \subseteq [n]$ of size $n^{1-O(\sqrt{\log n})} = n^{1-o(1)}$ such that $A$ contains no tuple $(x_1, x_2, \dots, x_k)$ of not all equal integers satisfying $\sum_{i=1}^{k-1} x_i = (k-1)x_k$.
Let $n_0\leq n$ be the maximum integer with $t\binom{k+1}{2}\mid n_0$ and let $n_1=n_0/t=\binom{k+1}{2}N$.

We first construct an $n_1$-vertex graph $H$ as follows. Partition $V(H)$ into $k$ subsets $V_1,V_2,\ldots,V_k$ and set  $V_1 = [N], V_2 = [2N], \dots, V_k = [kN]$. Edges in $H$ are allowed only between consecutive parts $V_i$ and $V_{i+1}$ for $i \in [k-1]$, as well as between $V_1$ and $V_k$. For any vertices $x_1 \in V_1, x_2 \in V_2, \dots, x_k \in V_k$, we have $x_i x_{i+1} \in E(H)$ for $i \in [k-1]$ if and only if $x_{i+1} - x_i \in A$, and $x_1 x_k \in E(H)$ if and only if $(x_k - x_1)/(k-1) \in A$.

It is clear that every $k$-cycle in $H$ meets each $V_i$ in exactly one vertex, and each vertex of $V_1$ is contained in $|A|$ distinct $k$-cycles. Hence $H$ contains at least $|V_1| \cdot |A| = n_1^{2-o(1)}$ copies of $C_k$.

For a $k$-cycle $x_1 x_2 \dots x_k x_1$ in $H$, set $c_1 = x_2 - x_1$, $c_2 = x_3 - x_2$, $\dots$, $c_{k-1} = x_k - x_{k-1}$, and $c_k = (x_k - x_1)/(k-1)$. By construction, each $c_i$ belongs to $A$. Moreover, a direct computation gives $\sum_{i=1}^{k-1} c_i = (k-1) c_k$, which forces $c_1 = c_2 = \dots = c_k$ by the property of $A$. Consequently, each edge between $V_1$ and $V_2$ lies in exactly one $k$-cycle. Therefore $H$ contains $n_1^{2-o(1)}$ edge‑disjoint $k$-cycles (each regarded as a monochromatic cycle of a unique color). Since every $k$-cycle in $H$ must contain an edge between $V_1$ and $V_2$, no rainbow $k$-cycle exists in $H$. Hence $ex_{C_k}(n_1, C_k) \geq n^{2-o(1)}$.

Now we show $ex_{C_k(t)}(n, C_k) \geq n^{2-o(1)}$. Let $H'$ denote the subgraph of $H$ consisting of $n_1^{2-o(1)}$ monochromatic $C_k$-copies. Consider the $t$-blow-up $H'(t)$ of $H'$, where each vertex $v \in V(H')$ is replaced by a set $X_v$. 
Then $v(H'(t))=n_0$. 
For an edge $ab \in E(H'(t))$ with $a \in X_u$ and $b \in X_v$, assign to $ab$ the same color as the edge $uv$ in $H'$. Consequently, $H'(t)$ decomposes into 
$$ex_{C_k}(n_0, C_k)  = n_0^{2-o(1)} = n^{2-o(1)}$$
edge-disjoint copies of $C_k(t)$. It remains to verify that $H'(t)$ is rainbow $C_k$-free. Suppose to the contrary that $G^*$ is a rainbow copy of $C_k$ in $H'(t)$. For each vertex $u \in V(G^*)$, let $u'$ denote the unique vertex in $V(H')$ such that $u \in X_{u'}$, and set $U = \{ u' : u \in V(G^*) \}$. If $|U| = k$, then $H'[U]$ contains a rainbow copy of $C_k$, contradicting the choice of $H'$. If $|U| <k$, then $H'[U]$ contains a shorter odd cycle than $C_k$. However, Since $H'$ is a blow-up of $C_k$, $g_o(H')=k$, a contradiction.
\hfill$\square$ \medskip

\section{Proof of Theorem \ref{Thm-c4}}

To establish the foundation for the subsequent proof, we first introduce some necessary terminology.  
A vertex $w$ is called a {\em rainbow common neighbor} of two distinct vertices $x$ and $y$ if $w$ is adjacent to both $x$ and $y$, and the colors on the edges $wx$ and $wy$ are distinct.
Based on the number of such neighbors, we classify an unordered pair of distinct vertices $\{u, v\}$ as follows: If $\{u, v\}$ has three rainbow common neighbors, it is called a {\em large vertex pair}; if $\{u, v\}$ has two rainbow common neighbors, it is called a {\em medium vertex pair}; if $\{u, v\}$ has at most one rainbow common neighbor, it is called a {\em small vertex pair}.

For an edge-colored graph $G$ and an edge $e\in E(G)$, we use $c(e)$ to denote the color of edge $e$, and use $C(G)$ to denote the set of colors appearing in $G$. For two edge-colored graphs $G_1$ and $G_2$, we say $G_1$ and $G_2$ have the same {\em coloring pattern} if
\begin{itemize}
    \item $G_1$ is isomorphic to $G_2$, and 
    \item there exist an isomorphism $\phi:V(G_1)\rightarrow V(G_2)$ and a bijection $f:C(G_1)\rightarrow C(G_2)$ such that for any edge $uv$ of $G_1$, $f(c(uv))=c(\phi(u)\phi(v))$.
\end{itemize}
For a large vertex pair $\{u,v\}$ (assume the three rainbow common neighbors of $u,v$ are $x_1,x_2,x_3$), we use $B^{u,v}$ denote the complete bipartite graph with parts $\{u,v\}$ and $\{x_1,x_2,x_3\}$.

\begin{figure}[ht]
    \centering
    \includegraphics[width=250pt]{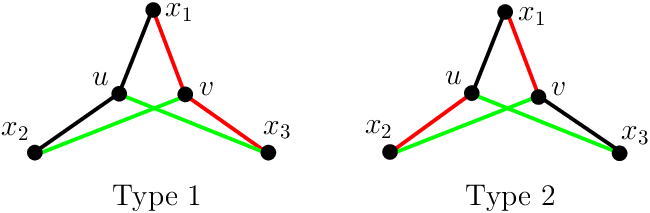}\\
    \caption{Two types of coloring patterns of $B^{u,v}$.} \label{Types}
\end{figure}

\begin{lemma}\label{lem-5-3}
Let $G$ be a rainbow $C_4$-free graph consisting of monochromatic $C_k$-copies.
Then for any two vertices $u,v$ of $G$, 
\begin{itemize}
    \item if $k\geq 6$, then $u,v$ have at most six rainbow common neighbors; if $k=4,5$, then $u,v$ have at most three rainbow common neighbors.
    \item  If $u,v$ have exactly three rainbow common neighbors, then
    \begin{itemize}
      \item [(1).] if $k=5$, then the coloring pattern of $B^{u,v}$ is Type 1 or 2 (see Figure \ref{Types});
      \item [(2).] if $k=4$, then the coloring pattern of $B^{u,v}$ is Type 1 (see Figure \ref{Types}).
    \end{itemize}
\end{itemize}
\end{lemma}

\begin{proof}
    Suppose $u$ and $v$ share $d$ rainbow common neighbors, denoted $x_1, x_2, \ldots, x_d$.
Without loss of generality, assume that $c(ux_1)$ is black and $c(vx_1)$ is red.

Since $G$ is rainbow $C_4$-free, for each $i \in \{2,3,\ldots,d\}$, we have $\{c(ux_i), c(vx_i)\} \cap \{\text{red,black}\} \neq \emptyset$.
Given that each of $u$ and $v$ is incident to exactly four edges colored either black or red, and $ux_1$, $vx_1$ account for two of these, it follows that $d \leq 7$.

Suppose, for contradiction, that $d = 7$. Then all rainbow common neighbors appear as depicted in Figure \ref{geq6}(1).
In this case, $c(ux_4) = c(ux_5)$ must be a new color, say blue. By the same reasoning, $c(ux_3)$ is also blue. Consequently, $u$ is incident to three blue edges, a contradiction. Hence, $d \leq 6$.

Now assume $k = 4,5$ and $d \geq 3$. Define two stars $S_u = G[ux_1,\ldots,ux_d]$ and $S_v = G[vx_1,\ldots,vx_d]$.

\noindent{\bf Case 1.} $S_u$ or $S_v$ is not rainbow.

By symmetry, we may assume $S_u$ is not rainbow and that $c(ux_2) =$ black.
Choose an integer $p$ with $3\leq p\leq d$ arbitrary, and let $c(vx_p)=\alpha$. Clearly, $\alpha=$black; otherwise, $k = 5$ and the black $C_5$ is $x_px_1ux_2vx_p$, contradicting that $x_2$ is a rainbow common neighbor of $u$ and $v$. 

If there exists such integer $p$ such that $\alpha$ is neither black nor red (say $\alpha=$ blue), then since $ux_1vx_pu$ is not a rainbow $C_4$, we must have $c(ux_p)=$red.
Since $ux_2vx_pu$ is not a rainbow $C_4$, it follows that $c(vx_2)$ is either red or blue.
If $c(vx_2)$ is red, then $k=5$, and the red $C_5$ and the black $C_5$ have the common edges, a contradiction.
Consequently, for each $3\leq p\leq d$, $c(vx_2)=c(vx_p)=\text{blue}$, implying 
$d=3$.
Moreover,
the coloring pattern of $B^{u,v}$ is Type 1.

Now we assume that for each $3\leq p\leq d$, $\alpha$=red.
Since $c(vx_1)$ is also red, it follows that $d=3$.
For the edges $ux_3$ and $vx_2$, it is clear that $\{c(ux_3), c(vx_2)\} \cap \{\text{red},\text{black}\} = \emptyset$ (for otherwise, if $c(vx_2)=\text{red}$ or $c(ux_3)=\text{black}$, then $u$ is incident to three black edges or $v$ is incident to three red edges, a contradiction).
Moreover, we have $c(vx_2) = c(ux_3)$, which must be a new color (say green); otherwise $ux_3vx_2u$ would form a rainbow $C_4$, a contradiction.
Consequently, the coloring pattern of $B^{u,v}$ is Type 1.
Furthermore, $k = 5$; otherwise, if $k=4$, then $uv$ belongs to both a green and a red $C_4$, a contradiction.

\noindent{\bf Case 2.} Both $S_u$ and $S_v$ are rainbow.

Without loss of generality, let $c(ux_2) =$ red and $c(vx_3) =$ black.
Since both $S_u$ and $S_v$ are rainbow,
$\{c(ux_3), c(vx_2)\} \cap \{\text{red},\text{black}\} = \emptyset$ and $d=3$.

Since $ux_3vx_2u$ is not a rainbow $C_4$ and $\{c(ux_3), c(vx_2)\} \cap \{\text{red},\text{black}\} = \emptyset$, we must have $c(ux_3) = c(vx_2)$, which is a new color (say green).
Thus, for $k = 4,5$, the coloring pattern of $B^{u,v}$ is Type 2.
Furthermore, $k = 5$; otherwise, if $k=4$, then $uv$ belongs to both a green and a red $C_4$, a contradiction.
\end{proof}

\begin{figure}[ht]
    \centering
\includegraphics[width=300pt]{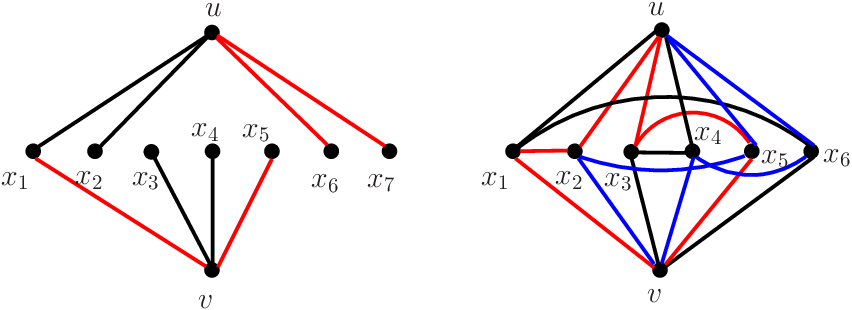}\\
    \caption{Rainbow common neighbors of $u,v$ for the case $d=7$ (left side). In the right side figure, $u,v$ have six rainbow common neighbors, this demonstrates the sharpness of $d$ for the case $k=6$.} \label{geq6}
\end{figure}

The following lemma helps estimate the
number of large, medium, and small vertex pairs.

\begin{lemma}\label{lem-2-1}
If $G$ is a rainbow $C_4$-free graph consisting of monochromatic $C_4$-copies, then
\begin{itemize}
    \item for any large vertex pair $\{u,v\}$, two of the three pairs in their rainbow common neighbors are small vertex pairs, and
    \item for any small vertex pair $\{x,y\}$, there are at most one large vertex pair whose rainbow common neighborhoods contain $\{x,y\}$.
\end{itemize}
\end{lemma}
\begin{proof}
    Let $\{u,v\}$ be a large vertex pair and $x_1,x_2,x_3$ be their rainbow common neighbors.
    By Lemma \ref{lem-5-3}, the coloring pattern is of Type 1.
    Furthermore, there exist two additional vertices $y_1$ and $y_2$ such that the union of the three monochromatic $C_4$s has the form shown in Figure \ref{c4-1}.

\begin{figure}[ht]
    \centering   \includegraphics[width=130pt]{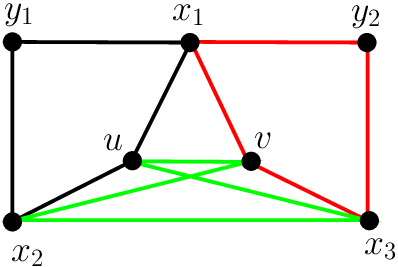}\\
    \caption{The union of three  monochromatic $C_4$s. } \label{c4-1}
\end{figure}

    We complete the first statement by showing $\{x_1,x_3\}$ and $\{x_1,x_2\}$ are small vertex pairs.
    By symmetry, we only need to show that $\{x_1,x_3\}$ is a small vertex pair.
    Note that $u$ is a rainbow common neighbor of $\{x_1,x_3\}$.
    Assume for the contrary that $z\neq u$ is also a rainbow common neighbor of $\{x_1,x_3\}$.
    Since $zx_1ux_3z$ is not a rainbow $C_4$, it follows that $\{c(zx_1),c(zx_3)\}\cap\{\text{black,green}\}\neq \emptyset$.
    Hence, $z\in \{x_2,y_1\}$.
    If $z=x_2$, then $x_2x_1\in E(G)$ and $c(x_2x_1)\notin\{\text{red,blue,green}\}$.
    However, $x_1x_2uvx_1$ is a rainbow $C_4$, a contradiction.
    Similarly, if $z=y_1$, then $y_1x_2vx_3y_1$ is a rainbow $C_4$, a contradiction.

    Next, we show the second statement.
    Suppose to the contrary that $\{x,y\}$ is a small vertex pair that is contained in the rainbow common neighborhoods of two different large vertex pairs $\{u,v\}$ and $\{u',v'\}$. Assume that the coloring pattern of $B^{u,v}$ is shown in Figure \ref{Types}(1).
    Then either $\{x,y\}=\{x_1,x_2\}$ or $\{x,y\}=\{x_1,x_3\}$ (by symmetry, say $x=x_1$ and $y=x_2$).
    If $\{u,v\}\cap \{u',v'\}=\emptyset$, then one of $\{u,v\}$ and one of $\{u',v'\}$ are rainbow common neighbors of $x_1,x_2$, a contradiction.
    Hence, $\{u,v\}\cap \{u',v'\}\neq \emptyset$ (without loss of generality, let $v=v'$).


Since $x_1$ and $x_2$ are rainbow common neighbors of $u'$ and $v$, we have $\{\text{green},\text{red}\} \subseteq C(B^{u',v})$.
By the characterization of coloring pattern Type 1, $u'$ lies in either the red or the green $C_5$.
Hence $u' \in \{x_3, y_2\}$ (see Figure \ref{c4-1}).
In either case, one of $x_1$ or $x_2$ is not a rainbow common neighbor of $u'$ and $v$, a contradiction.
\end{proof}

\noindent{\bf Proof of Theorem \ref{Thm-c4}:}
For $k\geq 5$, let $G$ be an $n$-vertex rainbow $C_4$-free graph consisting of $ex_{C_k}(n,C_4)$ monochromatic $C_k$-copies.
Let
$$
d(k)=\left\{
\begin{array}{ll}
6,   &k\geq 6;\\
3, & k=5.\\
\end{array}
\right.$$
For any two vertices $u,v$ of $G$, there are at most $d(k)$ rainbow common neighbors in $G$.
Hence, 
$$\# \mbox{ rainbow }K_{1,2}\leq d(k){n\choose 2}.$$
On the other hand, since each $v\in V(G)$ belongs to exactly $d_G(v)/2$ monochromatic $C_k$s,
$$\#\mbox{ Rainbow }K_{1,2}=4\sum_{v\in G} {d(v)/2\choose 2} =\sum_{v\in V(G)}\left(\frac{d(v)^2}{2}-d(v)\right)\geq \frac{1}{2n}\left(\sum_{v\in G}d(v)\right)^2-2e=\frac{2e^2}{n}-2e.$$
Hence, $\frac{2e^2}{n}-2e\leq d(k){n\choose 2}$.
If $k\geq 6$, then 
$e(G)\leq (\sqrt{6}n^{3/2}+n)/2$, and hence 
$$ex_{C_k}(n,C_4)=e(G)/k\leq\frac{\sqrt{6}n^{3/2}+n}{2k}.$$
If $k=5$, then $e(G)\leq (\sqrt{3}n^{3/2}+n)/2$, and hence 
$$ex_{C_5}(n,C_4)=e(G)/5\leq\frac{\sqrt{3}n^{3/2}+n}{10}.$$

For $k=4$, we will give a more subtle discussion.
Let 
$\mathcal{P}$ be the set of $(P_s,P_\ell)$, where $P_s$ is a small vertex pair and $P_\ell$  is a large vertex pair whose rainbow common neighborhood contains both vertices in $P_s$.
Assume that there are $s$ small vertex pairs, $m$ medium vertex pairs and $\ell$ large vertex pairs in $G$.
Then $s+m+\ell={n\choose 2}$.
The first statement of Lemma \ref{lem-2-1} implies $|\mathcal{P}|=2\ell$, and the second statement of Lemma \ref{lem-2-1} implies $|\mathcal{P}|\leq s$.
Hence, $\ell\leq s/2$.
Since each large vertex pair has three rainbow common neighbors, each medium vertex pair has  two rainbow common neighbors and each small vertex pair has at most one rainbow common neighbors, it follows that 
$$\#\mbox{ rainbow }K_{1,2}\leq 3\ell+2m+s=2{n\choose 2}+\ell-s\leq 2{n\choose 2}.$$
Consequently, $2{n\choose 2}\geq \frac{2e^2}{n}-2e$, implying $e(G)\leq (n+\sqrt{2}n^{3/2})/2$ and
$$ex_{C_4}(n,C_4)=e(G)/4\leq \frac{n+\sqrt{2}n^{3/2}}{8}.$$

We now construct a lower bound.
Let $n_0=\lfloor n/2\rfloor$ and let $G_0$ be an $n_0$-vertex $C_4$-free graph with
$$
e(G_0) = {ex}(n', C_4)=\frac{1}{2}n_0^{3/2} + o(n_0^{3/2}).$$
Define $G'$ as the $2$-blow-up $G_0(2)$ of $G_0$. Then $n-1\leq |G'| = n$. In this construction, each edge of $G_0$ is replaced by a $C_4$ in $G'$, and each vertex $v \in V(G_0)$ is replaced by a $2$-set, denoted by $X_v$. Clearly, $G'$ consists of 
$$ex(n_0,C_4)=\frac{1}{2}n_0^{3/2} + o(n_0^{3/2})=\frac{\sqrt{2}n^{3/2}}{8}+o(n^{3/2})$$ monochromatic $C_4$-copies. Assign a distinct color to each of these $C_4$'s.
Next, we show that $G'$ contains no rainbow $C_4$.
Suppose, to the contrary, that $G'$ contains a rainbow $C_4$, and denote it by $a_1a_2a_3a_4a_1$.
Then there exist vertices $v_1, v_2, v_3, v_4$ in $G_0$ such that $v_i \in X_{a_i}$ for each $i \in [4]$.
It is easy to verify that $v_1, v_2, v_3, v_4$ are distinct, and hence $v_1v_2v_3v_4v_1$ is a $C_4$ in $G_0$, a contradiction.
Therefore,
$$ex_{C_4}(n,C_4)\geq ex(n/2,C_4)\geq \frac{\sqrt{2}n^{3/2}}{8}+o(n^{3/2}).$$
\hfill$\square$ \medskip

\section{Concluding remarks}

In this paper, we focused mainly on the $F$-multicolor Tur\'an number in the case where both $F$ and $G$ are cycles. In this setting, together with the previously known characterization of the upper bound case, our results give a complete description of when $ex_F(n,G)$ attains each of the three natural thresholds: the upper bound from Ineq.~(\ref{upper-lower}), the lower bound from Ineq.~(\ref{upper-lower}), and the $(6,3)$-type bound $n^{2-o(1)}$. For general graph pairs $(F,G)$, however, many basic questions remain open.

\smallskip

\noindent{\bf 1. Graph pairs attaining the lower bound form Ineq.~(\ref{upper-lower}).}
Following \cite{Imolay}, a fundamental open problem is to characterize all graph pairs $(F,G)$ for which
$$ex_F(n,G)=\frac{n^2}{v(F)^2}+o(n^2).$$
This problem is difficult even in seemingly simple cases. Although Theorem \ref{thm-1} does not provide a complete characterization, it gives an effective exclusion method for determining whether the multicolor Tur\'{a}n number attains the lower bound form Ineq.~(\ref{upper-lower}).

\begin{figure}[ht]
    \centering   \includegraphics[width=250pt]{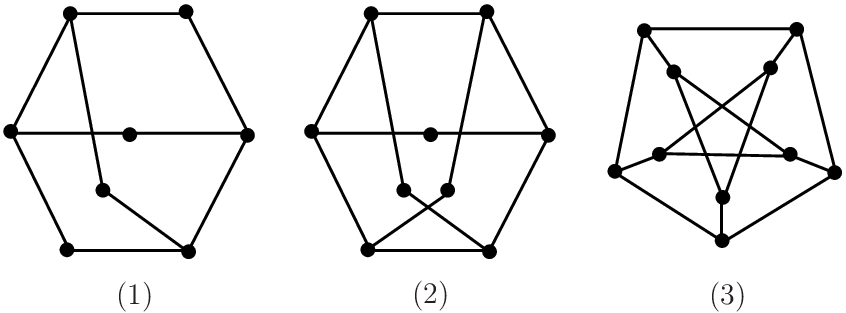}\\
    \caption{Three possible graphs of $F$ for $ex_F(n,C_3)$ attaining the lower bound. } \label{leq-ten}
\end{figure}

For example, when $G=C_3$, Theorem \ref{thm-1} provides a practical way to screen candidate graphs $F$. Since a nondegenerate pair $(F,C_3)$ must admit no homomorphism from $C_3$ to $F$, any such graph $F$ is necessarily triangle-free; if we further restrict to $C_4$-free graphs, then condition (A) in Theorem \ref{thm-1} is automatic for every nonadjacent pair. Thus one is naturally led to search among graphs with no $3$-cycles and no $4$-cycles for those that avoid the obstruction given by Theorem \ref{thm-1}. With the help of programming, among graphs on at most ten vertices, besides $F=C_5$, there are only three additional candidates that survive this test (see Figure \ref{leq-ten}). Whether any of these candidates actually attains the lower bound remains open.

\smallskip

\noindent{\bf 2. The nondegenerate case when both $F$ and $G$ are cycles.}
When $F=C_{2k+1}$ and $G=C_{2\ell+1}$ with $k>\ell$, the first nontrivial case is $k=\ell+1$, which is settled in our paper by the asymptotic formula
$$ex_{C_{2\ell+3}}(n,C_{2\ell+1})=\frac{n^2}{(2\ell+3)^2}+o(n^2).$$
This is closely parallel to the corresponding generalized Tur\'an problem for adjacent odd cycles, where balanced $C_{2\ell+3}$-blow-up plays the extremal role \cite{Beke}.

However, once $k>\ell+1$, the picture becomes much subtler. In the generalized Tur\'an setting, recent work \cite{Beke} shows that when the longer odd cycle is sufficiently large compared with the shorter one, the balanced $C_{2\ell+3}$-blow-up   is no longer asymptotically optimal. 
In the multicolor Tur\'an setting, Kov\'{a}cs and Nagy \cite{Kovacs} also constructed an unbalanced $C_5$-blow-up, whose $C_{2k+1}$-packing number far exceeds that of the balanced $C_5$-blow up. A similar construction shows that this phenomenon also occurs in the case $k > \ell+1 > 3$.
Determining the correct dominant term of $ex_{C_{2k+1}}(n,C_{2\ell+1})$ for $k>\ell+1$, as well as that of the generalized Tur\'an number $ex(n,C_{2k+1},C_{2\ell+1})$, therefore appears to be a highly challenging problem.

\smallskip

\noindent{\bf 3. Graph pairs attaining the $(6,3)$-type bound.}
Another important threshold is the $(6,3)$-type bound. Our results show that for graph pairs $(F,G)$ with $g_o(F)=g_o(G)=k$ and there exist homomorphisms from both $F$ and $G$ to $C_k$ (ensuring $\chi(F)=\chi(G)\leq 3$), then $ex_F(n,G)$ has order $n^{2-o(1)}$. It would be interesting to determine whether this condition is also necessary for $\chi(G)=3$.
More generally, it remains open whether there exists a graph pair $(F,G)$ with $\chi(G)\ge 4$ such that
$ex_F(n,G)=n^{2-o(1)}$.

\bigskip

\noindent{\bf Acknowledgments:}
Ping Li is supported by the National Natural Science Foundation of China No. 12201375. We are grateful to Suyun Jiang for recommending reference \cite{Beke}, which is helpful for the proof of  Theorem \ref{thm-3+}.

\end{document}